\documentclass[submission,copyright,creativecommons]{eptcs}
\providecommand{\event}{GandALF 2019} % Name of the event you are submitting to
\usepackage{breakurl}             % Not needed if you use pdflatex only.
\usepackage{underscore}           % Only needed if you use pdflatex.

\usepackage{amsmath}
\usepackage{amssymb}
\usepackage{graphicx}

\newtheorem{theorem}{Theorem}

\newtheorem{definition}[theorem]{Definition}
\newtheorem{lemma}[theorem]{Lemma}
\newtheorem{proposition}[theorem]{Proposition}
\newtheorem{corollary}[theorem]{Corollary}

\newtheorem{conjecture}{Conjecture}
\newenvironment{proof}{\emph{Proof:}$\\$}{$\\\Box\\$}

\title{Characterizing Strongly First Order Dependencies: The Non-Jumping Relativizable Case
}
\author{Pietro Galliani
\institute{Free University of Bozen-Bolzano, Italy}
\email{Pietro.Galliani@unibz.it}
}
\def\titlerunning{Non-Jumping Strongly First Order Dependencies}
\def\authorrunning{P. Galliani}

\newcommand{\tuple}{\mathbf}
\newcommand{\FO}{\texttt{FO}}

\newcommand{\M}{\mathfrak M}
\newcommand{\A}{\mathfrak A}
\newcommand{\B}{\mathfrak B}
\newcommand{\D}{\mathbf D} 
\newcommand{\DD}{\mathcal D}

\newcommand{\E}{\mathbf E}
\newcommand{\EE}{\mathcal E}
\newcommand{\F}{\mathbf F}

\newcommand{\nonempty}{\texttt{NE}}
\newcommand{\Dmax}{\mathbf D_{\textbf{max}}}

\newcommand{\Fmax}{\mathbf F_{\textbf{max}}}

\newcommand{\dom}{\texttt{Dom}}

\begin{document}
\maketitle

\begin{abstract}
Team Semantics generalizes Tarski's Semantics for First Order Logic by allowing formulas to be satisfied or not satisfied by sets of assignments rather than by single assignments. Because of this, in Team Semantics it is possible to extend the language of First Order Logic via new types of atomic formulas that express \emph{dependencies} between different assignments. 

Some of these extensions are much more expressive than First Order Logic proper; but the problem of which atoms can instead be added to First Order Logic without increasing its expressive power is still unsolved. 

In this work, I provide an answer to this question under the additional assumptions (true of most atoms studied so far) that the dependency atoms are \emph{relativizable} and \emph{non-jumping}. Furthermore, I show that the global (or Boolean) disjunction connective can be added to any strongly first order family of dependencies without increasing the expressive power, but that the same is not true in general for non strongly first order dependencies.  
\end{abstract}
\section{Introduction}
	Team Semantics \cite{hodges97,vaananen07} generalizes Tarski's Semantics for First Order Logic by letting formulas be satisfied or not satisfied by sets of assignments (called \emph{teams}) rather than just by single assignments. This semantics was originally developed by Hodges in \cite{hodges97} in order to provide a compositional semantics for Independence-Friendly Logic \cite{hintikkasandu89,mann11}, an extension of First Order Logic that generalizes its game-theoretic semantics by allowing agents to have \emph{imperfect information} regarding the current game position;\footnote{In \cite{cameron01}, a combinatorial argument was used to show that a compositional semantics cannot exist for Independence Friendly Logic if we require satisfaction (with respect to a model $\M$) to be a relation between single assignments and formulas. In \cite{galliani11}, this result was extended to the case of infinite models.} but, as observed by V\"a\"an\"anen \cite{vaananen07}, as a logical framework it deserves study in its own right.

	In the case of First Order Logic itself, this semantics is equivalent and reducible to the usual Tarskian semantics, but the higher order nature of its satisfaction relation makes it possible to extend it in new ways. This is of considerable theoretical interest: indeed, Team Semantics may be seen as a tool to describe and classify novel fragments of Second Order Logic, an issue of great importance --- and of deep connections, via Descriptive Complexity Theory, to the theory of computation --- regarding which much is still not known. It is also of more direct practical interest, because of the connections between Team Semantics and Database Theory (see for instance \cite{hannula2016finite,kontinen13}).  Probabilistic variants of Team Semantics have recently gathered attention (see for instance \cite{durand2018approximation,durand2018probabilistic,hannula2018facets}). 

Much of the initial wave of research in this area focused on specific Team Semantics-based extensions of First Order Logic, in particular Dependence Logic \cite{vaananen07} and later Independence Logic \cite{gradel13} and Inclusion Logic \cite{galliani12,gallhella13}; but there are still relatively few \emph{general} results regarding the effects of extending First Order Logic via Team Semantics.\footnote{Examples of results of this type can be found for instance in \cite{kontinen2014decidable,kontinen2016decidability}, which studies the complexity of the finite decidability problem in  First Order Logic plus generalized dependency atoms.}  The simplest way of doing so, for example, is by introducing \emph{generalized dependency atoms} $\D$ that express dependencies between different assignments; and it is a consequence of the higher order nature of Team Semantics that, even if $\D$ itself is first order definable (as a property of relations), the logic $\FO(\D)$ obtained by adding it to First Order Logic with Team Semantics may well be much more expressive than First Order Logic. 

	A natural question would then be: can we find necessary and sufficient conditions for that \emph{not} to happen? In other words, for which dependency atoms $\D$ or families of dependency atoms $\DD$ is it true that every sentence of $\FO(\D)$ (resp. $\FO(\DD)$, that is, the logic obtained by adding \emph{all} atoms $\D \in \DD$ to First Order Logic) is equivalent to some first order sentence? An answer to this would be of clear theoretical interest, as part of the before-mentioned programme of using Team Semantics to describe and classify fragments of Second Order Logic, since it would give us necessary and sufficient conditions for such fragments to be more expressive than First Order Logic; and it would also be of more practical interest, as it would allow us to find out which families of dependencies can be added to the language of First Order Logic while guaranteeing that all the convenient meta-logical properties of First Order Logic still hold.

This, however, has not been answered yet. In \cite{galliani2015upwards}, a very general family of dependencies was found that does not increase the expressive power of First Order Logic if added to it; but it is an open question whether any dependency that has this property is definable in terms of dependencies in that family (and, in fact, in this work we will show that this is false). 

Building on recent work in \cite{galliani19characterizing} on the classification of downwards closed dependencies, this work provides a partial answer to this under two additional assumptions, namely that such a dependency is \emph{relativizable} (Definition \ref{def:relativ}) and \emph{non-jumping} (Definition \ref{def:nonjump}). These are natural properties that are true of essentially all the strongly first order dependency atoms studied so far, and of most types of dependencies that are of interest; and thus, for those dependencies, the results of this work completely answer the above question. Additionally, a simple result concerning \emph{global} (or \emph{Boolean}) disjunctions in Team Semantics will be proved along the way --- as a necessary tool for the main result --- that may be seen as a preliminary step towards the study of such questions in the more general case of \emph{operators} (rather than mere atoms) in Team Semantics. 
\section{Preliminaries} 
In Team Semantics, formulas are satisfied or not satisfied by sets of assignments (called \emph{teams}) rather than by single assignments as in Tarskian semantics:
%
%The following definitions provide the basic framework of this semantics: 
\begin{definition}[Teams --- \cite{vaananen07}]
Let $\M$ be a first order model with domain $M$ and let $V$ be a set of variable symbols. Then a \emph{team} $X$ with domain $\dom(X) = V$ is a set of assignments $s: V \rightarrow M$. 
\end{definition}
\begin{definition}[Relation Corresponding to a Team --- \cite{vaananen07}]
Given a team $X$ and a tuple $\tuple v = v_1 \ldots v_k$ of variables occurring in its domain, we write $X(\tuple v)$ for the $k$-ary relation $\{(s(v_1) \ldots s(v_k)) : s \in X\}$. 
\end{definition}

\begin{definition}[Team Duplication --- \cite{vaananen07}]
Given a team $X$ over $\M$ and a tuple of pairwise distinct variables $\tuple y = y_1 \ldots y_k$ (which may or may not occur already in the domain of $X$), we write $X[M/\tuple y]$ for the team with domain $\dom(X) \cup \{y_1 \ldots y_k\}$ defined as 
\[
X[M/\tuple y] = \{s[m_1 \ldots m_k/y_1 \ldots y_k] : s \in X, (m_1 \ldots m_k) \in M^k\}
\]
where, as usual, $s[m_1 \ldots m_k/y_1 \ldots y_k]$ is the result of extending/modifying $s$ by assigning $m_1 \ldots m_k$ to $y_1 \ldots y_k$.
\end{definition}

\begin{definition}[Team Supplementation --- \cite{vaananen07}]
Given a team $X$ over $\M$, a tuple of distinct variables $\tuple y = y_1 \ldots y_k$ (which may or may not occur already in the domain of $X$) and a function $H: X \rightarrow \mathcal P(M^k) \backslash \{\emptyset\}$ assigning to each $s \in X$ a nonempty set of tuples of elements of $\M$, we write $X[H/\tuple y]$ for the team with domain $\dom(X) \cup \{y_1 \ldots y_k\}$ defined as 
\[
X[H/\tuple y] = \{s[m_1 \ldots m_k/y_1 \ldots y_k] : s \in X, (m_1 \ldots m_k) \in H(s)\}.
\]

As a special case of supplementation, if $\tuple a = a_1 \ldots a_k$ is a tuple of elements of the model we write $X[\tuple a/\tuple y]$ for $\{s[a_1 \ldots a_k/y_1 \ldots y_k]: s \in X\}$.
\end{definition}

\begin{definition}[Team Semantics for First Order Logic --- \cite{vaananen07}]
Let $\M$ be a first order model with at least two elements\footnote{We need at least two elements in our model in order to encode disjunctions in terms of existential quantifications in Proposition \ref{propo:emax} and Theorem \ref{thm:psitheta}. The case in which only one element exists is in any case trivial, and may be dealt with separately if required.}, let $\phi$ be a First Order formula over its signature in Negation Normal Form\footnote{As is common in the study of Team Semantics, we will generally assume that all expressions are in Negation Normal Form.}, and let $X$ be a team over $\M$ with domain containing the free variables of $\phi$. Then we say that $\phi$ is satisfied by $X$ in $\M$, and we write $\M \models_X \phi$, if this is a consequence of the following rules: 
\begin{description}
\item[TS-lit:] For all first order literals $\alpha$, $\M \models_X \alpha$ if and only if, for all $s \in X$, $\M \models_s \alpha$ in the usual sense of Tarskian Semantics; 
\item[TS-$\vee$:] For all $\psi_1$ and $\psi_2$, $\M \models_X \psi_1 \vee \psi_2$ iff there exist teams $Y_1, Y_2 \subseteq X$ such that $X = Y_1 \cup Y_2$,\footnote{We do not require $Y_1$ and $Y_2$ to be disjoint.} $\M \models_{Y_1} \psi_1$ and $\M \models_{Y_2} \psi_2$; 
\item[TS-$\wedge$:] For all $\psi_1$ and $\psi_2$, $\M \models_X \psi_1 \wedge \psi_2$ iff $\M \models_X \psi_1$ and $\M \models_X \psi_2$; 
\item[TS-$\exists$:] For all $\psi$ and all variables $v$, $\M \models_X \exists v \psi$ iff there exists some function $H:X \rightarrow \mathcal P(M)\backslash \{\emptyset\}$ such that $\M \models_{X[H/v]} \psi$; 
\item[TS-$\forall$:] For all $\psi$ and all variables $v$, $\M \models_X \forall v \psi$ iff $\M \models_{X[M/v]} \psi$. 
\end{description}

Given a sentence $\phi$ and a model $\M$ whose signature contains that of $\phi$, we say that $\phi$ is true in Team Semantics if and only if $\M \models_{\{\epsilon\}} \phi$, where $\{\epsilon\}$ is the team containing the only assignment $\epsilon$ over the empty set of variables. 
\label{def:teamsemantics}
\end{definition}

As mentioned in the Introduction, with respect to First Order Logic proper Team Semantics is equivalent and reducible to Tarskian Semantics. More precisely, it can be shown by structural induction that 
\begin{proposition}[\cite{vaananen07}]
For all first order formulas $\phi$, models $\M$ and teams $X$, $\M \models_X \phi$ if and only if, for all assignments $s \in X$, $\M \models_s \phi$ according to Tarskian Semantics. 

In particular, if $\phi$ is a first order sentence then $\phi$ is true in $\M$ in the sense of Team Semantics if and only if it is true in $\M$ in the sense of Tarskian Semantics. 
\label{propo:FOL}
\end{proposition}

What is then the point of Team Semantics? In brief, Team Semantics allows us to extend First Order Logic in new ways, like for instance by adding new types of atoms describing \emph{dependencies} between different assignments:

\begin{definition}[Generalized Dependency --- \cite{kuusisto2015}]
Let $k \in \mathbb N$. A $k$-ary \emph{generalized dependency} $\D$ is a class of models\footnote{Here and in the rest of the work, whenever the signature of a model is understood from the context to be of the form $\{R\}$ for some $k$-ary relation symbol $R$, we will write $(M, S)$ -- where $M$ is a set of elements and $S \subseteq M^k$ -- for the model $\M$ over this signature that has domain $M$ and whose interpretation $R^\M$ of the symbol $R$ is exactly $S$. When no ambiguity is possible we will also use the same letter for the relation and the relation symbol, writing e.g. $(M, R)$.} over the signature $\{R\}$, where $R$ is a $k$-ary relation symbol, that is closed under isomorphisms (that is, if $\M_1$ and $\M_2$ are isomorphic and $\M_1 \in \D$ then $\M_2 \in \D$ as well).
	Given a family of such dependencies $\DD = \{\D_1, \D_2, \ldots\}$, we write $\FO(\DD)$ for the language obtained by adding atoms of the form $\D_i \tuple y_i$ to
First Order Logic, where the $\tuple y_i$ range over all tuples of variables of the same arity as $\D_i$, with the satisfaction rules 
\begin{description}
\item[TS-$\D_i$:] $\M \models_X \D_i \tuple y_i$ if and only if $(M, X(\tuple y_i)) \in \D_i$. 
\end{description}
\end{definition}
A case of particular interest is the one in which the class of models describing the semantics of a generalized dependency is itself first order definable: 
\begin{definition}[First Order Generalized Dependency --- \cite{galliani2016strongly}]
A generalized dependency $\D$ is \emph{first order} if and only if there exists a first order sentence $\D(R)$, where $R$ is a relation symbol of the same arity as $\D$, such that 
    $(M, R) \in \D \Leftrightarrow (M, R) \models \D(R)$
for all models $(M, R)$. 
\label{def:FOdep}
\end{definition}
A peculiar aspect of Team Semantics is that, due to the second order existential quantification implicit in its rules for disjunction and existential quantification, first order generalized dependencies can still increase considerably the expressive power of First Order Logic when added to it. For example, the Team Semantics-based logics that have been most studied so far are Dependence Logic \cite{vaananen07}, Independence Logic \cite{gradel13} and Inclusion Logic \cite{galliani12}, that add to First Order Logic respectively 

\begin{description}
\item[Functional Dependence Atoms:] For all tuples of variables $\tuple x$ and $\tuple y$, $\M \models_X =\!\!(\tuple x; \tuple y)$ iff any two $s,s' \in X$ that agree on the value of $\tuple x$ also agree on the value of $\tuple y$; 
\item[Independence Atoms:] For all tuples of variables $\tuple x$, $\tuple y$ and $\tuple z$, $\M \models_X \tuple x \bot_\tuple y \tuple z$ iff for any two $s, s' \in X$ that agree on $\tuple y$ there is some $s'' \in X$ that agrees with $s$ on $\tuple x$ and $\tuple y$ and with $s'$ on $\tuple y$ and $\tuple z$;\footnote{As discussed in \cite{engstrom12}, this atom is closely related to database-theoretic \emph{multivalued dependencies}.} 
\item[Inclusion Atoms:] For all tuples of variables $\tuple x$ and $\tuple y$ of the same length, $\M \models_X \tuple x \subseteq \tuple y$ iff for all $s \in X$ there exists some $s' \in X$ with $s(\tuple x) = s'(\tuple y)$. 
\end{description}

It is easy to see that these three types of dependency atoms are all first order in the sense of Definition \ref{def:FOdep}. However, (Functional) Dependence Logic $\FO(=\!\!(\ldots; \cdot))$ is as expressive as full Existential Second Order Logic, and so is Independence Logic $\FO(\bot)$, whereas Inclusion Logic is equivalent to the positive fragment of Greatest Fixed Point Logic \cite{gallhella13} (and hence, by \cite{immerman82,vardi82}, it captures PTIME over finite ordered models). 

Does this imply that (Functional) Dependence Logic and Independence Logic are equivalent to each other and strictly contain Inclusion Logic? This is not as unambiguous a question as it may appear. It certainly is true that every Inclusion Logic sentence is equivalent to some Independence Logic sentence, that every Dependence Logic sentence is equivalent to some Independence Logic sentence, and that every Independence Logic sentence is equivalent to some Dependence Logic sentence; but on the other hand, it is \emph{not} true that every Inclusion Logic formula, or every Independence Logic one, is equivalent to some Dependence Logic formula. This follows at once from the following classification: 
\begin{definition}[Empty Team Property, Closure Properties --- \cite{vaananen07,galliani12,galliani2015upwards}]
Let $\D$ be a generalized dependency. Then 
\begin{itemize}
    \item $\D$ has the \textbf{Empty Team Property} iff $(M, \emptyset) \in \D$ for all $M$; 
    \item $\D$ is \textbf{Downwards Closed} iff whenever $(M, R) \in \D$ and $R' \subseteq R$ then $(M, R') \in \D$; 
    \item $\D$ is \textbf{Union Closed} iff whenever $\{R_i : i \in I\}$ is a family of relations over some $M$ such that $(M, R_i) \in \D$ for all $i \in I$ then $(M, \bigcup_i R_i) \in \D$;
    \item $\D$ is \textbf{Upwards Closed} iff whenever $(M, R) \in \D$ and $R \subseteq R'$ then $(M, R') \in \D$. 
\end{itemize}
\end{definition}
The first three of the above properties are preserved by Team Semantics:
 \begin{proposition}[Properties preserved by Team Semantics --- \cite{vaananen07,galliani12}]
Let $\DD$ be a family of dependencies, let $\phi(\tuple v) \in \FO(\DD)$ be a formula with free variables in $\tuple v$ and let $\M$ be a first order model. Then
\begin{itemize}
\item If all $\D \in \DD$ have the Empty Team Property then $\M \models_\emptyset \phi$;
\item If all $\D \in \DD$ are downwards closed and $\M \models_X \phi$ then $\M \models_{X'} \phi$ for all $X' \subseteq X$; 
\item If all $\D \in \DD$ are union closed and $\M \models_{X_i} \phi$ for all $i \in I$ then $\M \models_{\bigcup_i X_i} \phi$. 
\end{itemize}
	\label{propo:closure}
\end{proposition}
From these facts --- that are proven easily by structural induction --- it follows at once that functional dependence atoms (which are downwards closed, but not union closed) cannot be used to define inclusion atoms (which are union closed, but not downwards closed) or independence atoms (which are neither downwards closed nor union closed). Additionally, since all these three types of dependencies have the Empty Team Property we have at once that, even together, they cannot be used to define for instance the \emph{nonemptiness atom} $\nonempty = \{(M, P): P \not = \emptyset\}$, such that 
$\M \models_X \nonempty(v) \text { iff } X(v) \not = \emptyset$.\footnote{The choice of the variable $v$ is of course irrelevant here, and we could have defined $\nonempty$ as a $0$-ary dependency instead; but treating it as a $1$-ary dependency is formally simpler.}

Differently from functional dependence atoms, inclusion atoms and independence atoms, some types of generalized dependencies do not increase the expressive power of First Order Logic when added to it: this is the case, for example, of the $\nonempty$ dependency just introduced. More generally, it was shown in \cite{galliani2015upwards} that if $\mathcal D^\uparrow$ is the set of all \emph{upwards closed} first order dependencies and $=\!\!(\cdot)$ is the \emph{constancy atom} such that $\M \models_X =\!\!(v)$ iff $|X(v)| \leq 1$,\footnote{That is, $\M \models_X =\!\!(v)$ iff for all $s, s' \in X$, $s(v) = s'(v)$.} every sentence of $\FO(\DD^\uparrow, =\!\!(\cdot))$ is equivalent to some first order sentence. In other words, we have that $\DD^\uparrow \cup \{=\!\!(\cdot)\}$ is \emph{strongly first order} according to the following definition: 
\begin{definition}[Strongly First Order Dependencies --- \cite{galliani2016strongly}]
A dependency $\D$, or a family of dependencies $\DD$, is said to be \emph{strongly first order} iff every sentence of $\FO(\D)$ (resp. $\FO(\DD)$) is equivalent to some first order sentence.\footnote{It is worth pointing out here that if $\D$ is strongly first order then it is first order in the sense of Definition \ref{def:FOdep}, because $(M, R) \in \D \Leftrightarrow (M, R) \models \forall \tuple x (\lnot R \tuple x \vee (R \tuple x \wedge \D \tuple x))$. The converse is however not true in general.} 
\end{definition}

Additionally, it is clear that any dependency $\E$ that is definable in $\FO(\DD^\uparrow, =\!\!(\cdot))$, in the sense that there exists some formula $\phi(\tuple v) \in \FO(\DD^\uparrow, =\!\!(\cdot))$ over the empty signature such that $\M \models_X \E \tuple v \Leftrightarrow \M \models_X \phi(\tuple v)$, is itself strongly first order. This can be used, as discussed in \cite{galliani2015upwards}, to show that for instance the negated inclusion atoms 
\[
\M \models_X \tuple x \not \subseteq \tuple y \text{ iff } X(\tuple x) \not \subseteq X(\tuple y)
\]
are strongly first order, as they can be defined in terms of upwards closed first order dependencies and constancy atoms; and as mentioned in \cite{galliani2016strongly}, the same type of argument can be used to show that all first order dependencies $\D(R)$ where $R$ has arity one are also strongly first order. 

This led to the following 

\begin{conjecture}[\cite{galliani18safe}]
	Every strongly first order dependency $\D(R)$ is definable in terms of upwards closed dependencies and constancy atoms.
\label{conj:1}
\end{conjecture}
In the next section we will show, by a simple argument, that this conjecture is not true as stated; but that it can be recovered (albeit not yet proved) by adding an additional and commonly used connective to the language of Team Semantics. 

We also recall here the following generalization of the notion of strongly first order dependency: 

\begin{definition}[Safe Dependencies \cite{galliani18safe}]
Let $\DD$ and $\EE$ be two families of dependencies. Then we say that $\DD$ is \emph{safe} for $\EE$ iff any sentence of $\FO(\DD, \EE)$ is equivalent to some sentence of $\FO(\EE)$. 
\end{definition}
Clearly, a dependency is strongly first order if and only if it is safe for the empty set of dependencies. However, as shown in \cite{galliani18safe}, a strongly first order dependency is not necessarily safe for \emph{all} families of dependencies: in particular, the constancy atom is not safe for the unary\footnote{This is a binary first order dependency, defined by the sentence $\D(R) = \forall x y (R x y \rightarrow \exists z R z x)$. The term ``unary'' is used here because each ``side'' of the dependency may have only one variable.} inclusion atom $v_1 \subseteq v_2$, in which $v_1$ and $v_2$ must be single variables (rather than tuples of variables).\footnote{As a quick aside, similar phenomena occur in the study of the theory of second-order generalized quantifiers \cite{kontinen2010definability}. This suggests the existence of interesting --- and, so far, largely unexplored --- connections between the theory of second order generalized quantifiers and that of generalized dependency atoms.} On the other hand, in \cite{galliani19characterizing} it was shown that strongly first order dependencies are safe for any family of downwards closed dependencies: 
\begin{theorem}[\cite{galliani19characterizing}, Theorem 3.8]
Let $\DD$ be a family of strongly first order dependencies and let $\EE$ be a family of downwards closed dependencies. Then every sentence of $\FO(\DD, \EE)$ is equivalent to some sentence of $\FO(\EE)$.
\end{theorem}

It is also worth mentioning here that the same notions of safety and strong first orderness can be easily generalized to \emph{operators}. For example, in \cite{galliani18safe} it was shown that the \emph{possibility operator} 
\[
    \M \models_X \diamond \phi \text{ iff } \exists Y \subseteq X, Y \not = \emptyset, \text{ s.t. } \M \models_Y \phi
\]
is safe for any collection of dependencies $\DD$, in the sense that every sentence of $\FO(\DD, \diamond)$ is equivalent to some sentence of $\FO(\DD)$. In the next section, we will instead see an example of an operator that is safe for any \emph{strongly first order} collection of dependencies, but that is not safe for some other (non strongly first order, albeit still first order) dependency families. 
\section{The Unsafety and Necessity of Global Disjunction}
A connective often added to the language of Team Semantics is the \emph{global} (or \emph{Boolean})\footnote{The term ``Boolean disjunction'' is most common in the literature, but it may give the wrong impression: as far as the author knows, there is no particular relation between this connective and Boolean algebras. ``Classical Disjunction'' is also a term sometimes used, because of the analogy between this semantics and the usual rule for disjunction; but this may also be misleading, because if we replace $\vee$ with $\sqcup$ in Team Semantics the truth conditions of First Order sentences become different from the usual ones (for instance, $\forall x \forall y (x = y \vee x \not = y)$ is always true, but $\forall x \forall y (x = y \sqcup x \not = y)$ is false in any model with at least two elements). In this work the term ``Global Disjunction'' will be used instead, to emphasize that when evaluating $\sqcup$ the \emph{whole} current team must satisfy one of the disjuncts.} disjunction 

\begin{description}
\item[TS-$\sqcup$:] $\M \models_X \phi \sqcup \psi$ if and only if $\M \models_X \phi$ or $\M \models_X \psi$.
\end{description}

This is different from the disjunction $\vee$ of Definition \ref{def:teamsemantics}: for example, a team $X$ of the form $\{(v:0,w:0),(v:0,w:1)\}$ does not satisfy $v=w \sqcup v \not = w$, although it satisfies $v=w \vee v \not = w$.

It is well known in the literature that, as long as the Empty Team Property holds in our language and the model contains at least two elements, this connective can be expressed in terms of constancy atoms as 
\[
\phi \sqcup \psi \equiv \exists p q (=\!\!(p) \wedge =\!\!(q) \wedge ( (p=q \wedge \phi) \vee (p \not = q \wedge \psi))),
\]
where $p$ and $q$ are two new variables not occurring in $\phi$ or $\psi$. However, this is not enough to guarantee that this connective will not affect the expressive power of a language based on Team Semantics if added to it, because of two reasons: \begin{enumerate}
    \item The empty team property does not necessarily apply to all logics $\FO(\DD)$, and when it does not then the above definition is not necessarily correct; 
    \item As shown in \cite{galliani18safe} and recalled above, the constancy atom itself is not safe for all families of dependencies. 
\end{enumerate}
As we will now see, the following result nonetheless holds: 
\begin{proposition}[Global Disjunction is Safe for Strongly First Order dependencies]
Let $\DD$ be any\\strongly first order family of dependencies, and let $\FO(\DD, \sqcup)$ be the logic obtained by adding to $\FO(\DD)$ the $\sqcup$ connective with the semantics given above. Then every sentence of $\FO(\DD, \sqcup)$ is equivalent to some first order sentence. 
\label{propo:sqcup}
\end{proposition}
\begin{proof}
Let $\phi$ be any sentence of $\FO(\DD, \sqcup)$. Then apply iteratively the following, easily verified transformations 
\begin{itemize}
    \item $(\phi \sqcup \psi) \vee \theta \equiv \theta \vee (\phi \sqcup \psi) \equiv (\phi \vee \theta) \sqcup (\psi \vee \theta)$; 
    \item $(\phi \sqcup \psi) \wedge \theta \equiv \theta \wedge (\phi \sqcup \psi) \equiv (\phi \wedge \theta) \sqcup (\psi \wedge \theta)$; 
    \item $\exists v(\phi \sqcup \psi) \equiv (\exists v \phi) \sqcup (\exists v \psi)$;
    \item $\forall v(\phi \sqcup \psi) \equiv (\forall v \phi) \sqcup (\forall v \psi)$
\end{itemize}
until we obtain an expression $\phi'$, equivalent to $\phi$, of the form $\sqcup_i \psi_i$, where each $\psi_i$ is a sentence of $\FO(\DD)$.  But since $\DD$ is strongly first order, every such $\psi_i$ is equivalent to some first order sentence $\theta_i$; and, therefore, $\phi$ itself is equivalent to the first order sentence $\bigvee_i \theta_i$.
\end{proof}
Thus, whenever we have a family of strongly first order dependencies $\DD$ we can freely add the global disjunction connective $\sqcup$ to our language without increasing its expressive power. This is a deceptively simple result: in particular, it is not immediately obvious whether $\sqcup$ is similarly ``safe'' for families of dependencies $\DD$ that are \emph{not} strongly first order. In fact, this is not the case! To see why, let us first prove the following easy lemma: 

\begin{lemma}
Let $\phi \in \FO(\DD)$ be a formula in which the dependency atom $\D \in \DD$ occurs at least once, and suppose that $\M \models_X \phi$. Then there exists some $R$ such that $(M, R) \in \D$. Moreover, if $\M \models_\emptyset \phi$ then $(M, \emptyset) \in \D$ for all dependencies $\D$ appearing in $\phi$.
\label{lemma:Xtransmit}
\end{lemma}
\begin{proof}
The proof is a straightforward structural induction, but we report it in full for clarity's sake (note, as an aside, that if we added the global disjunction $\sqcup$ to our language the induction would not carry through). 
\begin{itemize}
	\item If $\phi$ is a literal, by assumption it must be of the form $\D \tuple v$ for some $\tuple v$. Then if $\M \models_X \phi$ then $(M, X(\tuple v)) \in \D$, as required; and moreover if $\M \models_X \phi$ for $X = \emptyset$ then $X(v) = \emptyset$ and hence $(M, \emptyset) \in \D$. 
\item If $\phi$ is of the form $\psi_1 \vee \psi_2$, some atom of the form $\D \tuple v$ must appear in $\psi_1$ or in $\psi_2$. Without loss of generality, let us suppose that it appears in $\psi_1$. Then, if $\M \models_X \phi$, we have that $X = Y \cup Z$ for two $Y$, $Z$ such that $\M \models_Y \psi_1$ and $\M \models_Z \psi_2$. But then, by induction hypothesis, there exists some $R$ such that $(M, R) \in \D$. Moreover, if $\M \models_\emptyset \phi$ and $\D$ is a dependency appearing in $\phi$ then necessarily $\M \models_\emptyset \psi_1$ and $\M \models_\emptyset \psi_2$, from which by induction hypothesis (since $\D$ appears in $\psi_1$ or in $\psi_2$) we have that $(M, \emptyset) \in \D$. 
\item If $\phi$ is of the form $\psi_1 \wedge \psi_2$ --- again, assuming without loss of generality that $\D$ appears in $\psi_1$ -- and $\M \models_X \phi$ then $\M \models_X \psi_1$ and $\M \models_X \psi_2$, and hence by induction hypothesis $(M, R) \in \D$. Moreover, if $\M \models_\emptyset \phi$ then $\M \models_\emptyset \psi_1$ and $\M \models_\emptyset \psi_2$, and so by induction hypothesis $(M, \emptyset) \in \D$ for all dependencies $\D$ appearing in $\phi$. 
\item If $\phi$ is of the form $\exists v \psi$ and $\M \models_X \phi$ then for some $Y = X[H/v]$ it is true that $\M \models_Y \psi$. Hence, by induction hypothesis, $(M, R) \in \D$ for some $R$. Moreover, if $\M \models_\emptyset \phi$ then $\M \models_\emptyset \psi$ (since $\emptyset[H/v] = \emptyset$ for all $H$) and hence by induction hypothesis $(M, \emptyset) \in \D$ for all dependencies $\D$ appearing in $\phi$. 
\item If $\phi$ is of the form $\forall v \psi$ and $\M \models_X \phi$ then for $Y = X[M/v]$ it is true that $\M \models_Y \psi$. Again, by induction hypothesis, this implies that $(M, R) \in \D$ for some $R$. Moreover, $\emptyset[M/v] = \emptyset$, and hence if $\M \models_\emptyset \phi$ and $\D$ appears in $\phi$ then $\M \models_\emptyset \psi$ and $(M, \emptyset) \in \D$ by induction hypothesis. 
\end{itemize}
\end{proof}

Now consider the two (first order, but not strongly first order) dependencies 
\begin{description}
    \item[TS-LO2:] $\M \models_X \textbf{LO2}(x,y,z)$ if and only if $X(xy)$ describes a total linear order with endpoints over all elements of $M$ and $X(z)$ does not contain the first element of this order, contains the second and the last, and whenever it contains an element it does not contain its successor (in the linear order) but it contains its successor's successor.
    \item[TS-LO3:] $\M \models_X \textbf{LO3}(x,y,z)$ if and only if $X(xy)$ describes a total linear order with endpoints over all elements of $M$ and $X(z)$ does not contain the first or the second elements of this order, contains the third and the last, and whenever it contains an element it does not contain its successor (in the linear order) nor its successor's successor but it contains its successor's successor's successor. 
\end{description}
Then the $\FO(\textbf{LO2}, \textbf{LO3}, \sqcup)$ sentence 
$(\exists x y z \textbf{LO2}(x,y,z)) \sqcup (\exists x y z \textbf{LO3}(x,y,z))$
is easily seen to hold in a model $\M$ if and only if the size $|M|$ of its domain is a multiple of two or of three (or it is infinite). 
However, there is no sentence  $\phi$ of $\FO(\textbf{LO2}, \textbf{LO3})$ that is true if and only if this property holds. Indeed, suppose that such a $\phi$ existed. Then the $\textbf{LO2}$ dependency cannot appear in it: indeed, for a model $\M_3$ with precisely three elements it must hold that $\M_3 \models_{\{\epsilon\}} \phi$, and if $\textbf{LO2}$ appeared in it then by Lemma \ref{lemma:Xtransmit} there should be some $R$ such that $(M_3, R) \in \textbf{LO2}$. But by definition, $(M_3, R) \not \in \textbf{LO2}$ for any $R$ (not even for $R = \emptyset$!), and so this cannot be the case. Similarly, $\textbf{LO3}$ cannot occur in $\phi$, because $\phi$ must be true in a model $\M_2$ with exactly two elements. 

Therefore $\phi$ must be first order; and a standard back-and-forth argument shows that there is no first order sentence over the empty signature that is true in a model if and only if its size is divisible by two or by three (or it is infinite).\\

Another consequence of Lemma \ref{lemma:Xtransmit} is that Conjecture \ref{conj:1} is false. Indeed, consider the dependency $\mathbf U$ such that 
	$\M \models_X \textbf{U} v$ if and only if $X(v) = \emptyset \text{ or } X(v) = M$. Then $\mathbf U v$ is satisfied by a team $X$ in a model $\M$ if and only if the variable $v$ takes no values at all (that is, $X$ itself is empty) or all possible values. This dependency is strongly first order by Proposition \ref{propo:sqcup}: indeed, it is easy to see that $\textbf U v \equiv \bot \sqcup \texttt{All}(v)$, where $\texttt{All}(v)$ is the first order and upwards closed (and, therefore, strongly first order) dependency such that $\M \models_X \texttt{All}(v)$ iff $X(v) = M$. However, $\mathbf U$ is not definable in Team Semantics (without global disjunction) in terms of constancy atoms and upwards closed first order dependencies. Indeed, suppose that $\mathbf U v$ is equivalent to $\phi(v)$ for some $\phi \in \FO(\DD^\uparrow, =(\cdot))$, where $\DD^\uparrow$ describes the family of all upwards closed first order dependencies. Then since for all $\M$ we have that $\M \models_\emptyset \mathbf U v$, by Lemma \ref{lemma:Xtransmit} $\M \models_\emptyset \D \tuple t$ for every occurrence $\D \tuple t$ of a dependence atom $\D$ in $\phi$. But if $\D$ is upwards closed and $\M \models_\emptyset \D \tuple t$ then $\M \models_Y \D \tuple t$ for all $Y$, and this is the case for all models $\M$. So, for any $\D \in \DD^\uparrow$, we can replace every occurrence $\D \tuple t$ of $\D$ in $\phi$ with $\top$ without affecting its satisfaction conditions. Therefore there is some $\psi(v) \in \FO(=\!\!(\cdot))$ which defines $\mathbf U v$, which is impossible since by Proposition \ref{propo:closure} all formulas in $\FO(=\!\!(\cdot))$ are downwards closed but $\mathbf U v$ is not. However, we can recover our conjecture by modifying it as follows: 

\begin{conjecture}
Every strongly first order dependency $\D(R)$ is definable in terms of upwards closed dependencies, constancy atoms and global disjunctions. 
\label{conj:1bis}
\end{conjecture}

In conclusion, even though we may add the $\sqcup$ operator ``for free'' as long as we are only working with strongly first order dependencies, this is not necessarily the case if we are working with more expressive types of dependencies; and even though Conjecture \ref{conj:1} is false, it may still be true if this operator is added to the basic language of Team Semantics. This suggests that the global (or Boolean) disjunction may have a more central role in the study of Team Semantics than the one it had so far. In the next section, we will see a proof of a special case of Conjecture \ref{conj:1bis}. 
\section{Non-Jumping, Relativizable Dependencies }
In this section we will prove a restricted version of Conjecture \ref{conj:1bis} under two additional (and commonly true) conditions. The first condition that we will assume will be that the dependencies we are discussing are \emph{relativizable} in the sense of the following definition:\footnote{This definition is related, but not identical, to the notion of relativization of formulas in Team Semantics discussed by R\"onnholm in \S 3.3.1 of \cite{ronnholm2018arity}.}
\begin{definition}[Relativized Dependencies, Relativizable dependencies --- \cite{galliani19characterizing}, Definition 2.31]
Let $\DD$ be a family of dependencies and let $P$ be a unary predicate. Then the language $\FO(\DD^{(P)})$ adds to First Order Logic the \emph{relativized dependence atoms} $\D^{(P)} \tuple y$ for all $\D \in \DD$, and the corresponding semantics (for models whose signature contains $P$) is given by 
\[
\M \models_X \D^{(P)} \tuple y \text{ iff } (P^{\M}, X(\tuple y)) \in \D
\]
where $P^{\M}$ is the interpretation of $P$ in $\M$.\footnote{In this work, when no ambiguity is possible we will generally write the relation symbols $P$, $R$, $S$ instead of the corresponding interpretations $P^\M$, $R^\M$, $S^\M$.}

A dependency $\D$, or a family of dependencies $\DD$, is said to be \emph{relativizable} if any sentence of $\FO(\D^{(P)})$ (resp. $\FO(\DD^{(P)})$) is equivalent to some sentence of $\FO(\D)$ (resp. $\FO(\DD)$). 
\label{def:relativ}
\end{definition}

It follows easily from the above definition that if $\M \models_X \D^{(P)} \tuple y$ then $X(y_i) \subseteq P^{\M}$ for all $y_i \in \tuple y$: otherwise, it will always be the case that $(P^{\M}, X(\tuple y)) \not \in \D$. 
 
Essentially all the dependencies studied in the context of Team Semantics thus far are relativizable. Most of them have even the stronger property of being \emph{universe independent} in the sense of \cite{kontinen2014decidable}: in brief, whether $\M \models_X \D \tuple y$ or not depends only on the value of $X(\tuple y)$ (and not on the domain $M$ of $\M$), from which relativizability follows trivially.

As was pointed out to the author by Fausto Barbero in a personal communication, a counting argument shows that there exist generalized dependencies that are not relativizable. A concrete example is the unary  dependency $I_\infty = \{(M, P) : M \text{ is infinite }\}$. Of course this is not a first order dependency, and it is a unusual dependency in that whether $\M \models_X I_\infty v$ or not does not depend on $X(v)$ but only on $M$; but it is nonetheless a perfectly legitimate generalized dependency, and it is not relativizable. Indeed, the class of models $ \mathcal C = \{(M, P): P \text{ is infinite}\}$ is defined by the $\FO(I_\infty^{(P)})$ sentence $\exists v I_\infty^{(P)} v$; however,  the same class of models is not defined by any $\FO(I_\infty)$ sentence, because in any infinite model any occurrence of $I_\infty$ can be replaced by the trivially true literal $\top$ and the class $\mathcal C$ is not first order definable.
The author does not however know of any \emph{strongly first order} generalized dependency that is not relativizable. The following conjecture is, therefore, open and --- if true --- would allow us to remove the relativizability requirement: 
\begin{conjecture}
Every strongly first order generalized dependency is relativizable. 
\label{conj:2}
\end{conjecture}

In order to describe the second condition we need the following definition: 

\begin{definition}[\cite{galliani19characterizing}, Proposition 4.2]
Let $\D$ be any generalized dependency. Then $\Dmax$ is the dependency $\{(M, R): (M, R) \in \D \text{ and } \forall S \supsetneq R, (M, S) \not \in \D\}$.
\end{definition}

In general, for $\D$ first order there is no guarantee that whenever $(M, R) \in \D$ there is some $S \supseteq R$ at all such that $(M, S) \in \Dmax$; but, as we will see soon, if $\D$ is strongly first order this is indeed the case, and moreover $\Dmax$ itself is also strongly first order.

\begin{definition}
	A dependency $\D$ is \emph{non-jumping} if, for all sets of elements $M$ and all relations $R$ over $M$ of the same arity as $\D$, if $(M, R) \in \D$ then there exists some $R' \supseteq R$ such that 
	\begin{enumerate}
		\item $(M, R') \in \Dmax$; 
		\item For all relations $S$, if $R \subseteq S \subseteq R'$ then $(M, S) \in \D$. 
	\end{enumerate}
	\label{def:nonjump}
\end{definition}
In other words, a dependency $\D$ is non-jumping if whenever it holds of some $R$ we can ``enlarge'' $R$ to some $R'$ that is maximal among those that satisfy $\D$ and such that, furthermore, any relation $S$ between $R$ and $R'$ satisfies also $\D$. It is possible to find examples of jumping dependencies, like for instance $\D = \{(M, P) : |P| \not = 1\}$;\footnote{For this dependency, we have that $\M \models_X \D v$ if and only if $X = \emptyset$ or $|X(v)| \geq 2$. In other words, $\D v$ is equivalent to $\bot ~\sqcup~ \not =(v)$, where $\not =(v)$ is the (upwards closed) \emph{non-constancy} atom which is true in a team iff $|X(v)| \geq 2$. } but non-jumping dependencies nonetheless constitute a natural and general category of dependencies.

We now need to generalize the two following results from \cite{galliani19characterizing} to the case of dependencies that are not necessarily downwards closed: 
\begin{proposition}[\cite{galliani19characterizing}, Proposition 4.2 and Corollary 4.3] 
Let $\D$ be a downwards closed strongly first order dependency. Then $\Dmax$ is also strongly first order, and whenever $(M, R) \in \D$ there is some $R' \supseteq R$ such that $(M, R') \in \Dmax$.
\label{propo:dmax}
\end{proposition}
\begin{theorem}[\cite{galliani19characterizing}, Theorem 4.5]
	Let $\D$ be a downwards closed, strongly first order, relativizable dependency.\footnote{The empty team property is not required, because in the proof of Theorem 4.5 of \cite{galliani19characterizing} this property was necessary only to translate a global disjunction into $\FO(=\!\!(\cdot))$ and not to find the $\theta_i$ via the Chang-Makkai Theorem.} Then there are first order formulas $\theta_1(\tuple x, \tuple z)$ \ldots $\theta_n(\tuple x, \tuple z)$ over the empty signature such that, for all models $\M = (M, R)$, 
	\[
		\M \models \Dmax(R) \Rightarrow \M \models \bigvee_{i=1}^n (\exists \tuple z \forall \tuple x (R \tuple x \leftrightarrow \theta_i(\tuple x, \tuple z))).
	\]
	\label{thm:dmaxdef}
\end{theorem}
To do so, it suffices to observe the following: 
\begin{proposition}
Let $\D$ be a strongly first order dependency. Then there exists a downwards closed strongly first order dependency $\F$ such that $\D \subseteq \F$ and $\Fmax = \Dmax$. Moreover, if $\D$ is relativizable then so is $\F$.
\label{propo:emax}
\end{proposition}
\begin{proof}
		The dependency $\E \tuple x := \bot \sqcup (\forall p q \exists \tuple y ( (p \not = q \vee \tuple y = \tuple x) \wedge \D\tuple y))$ is strongly first order, because it is definable in terms of $\D$ and $\sqcup$ and because of Proposition \ref{propo:sqcup}, and $(M, R) \in \E$ iff $R = \emptyset$ or there is some $S \supseteq R$ such that $(M,S) \in \D$.\footnote{Note that the right disjunct alone would \emph{not} describe this property or be necessarily downwards closed. The problem is that, if we start from the empty team, we cannot ``enlarge'' it by choosing new values for $\tuple y$ for the assignments in which $p \not = q$ because no such assignments exist: so, if $\M \not \models_\emptyset \D \tuple x$, $\M \not \models_\emptyset (\forall p q \exists \tuple y ( (p \not = q \vee \tuple y = \tuple x) \wedge \D\tuple y)$ even though it may be the case that $\M \models_X \D \tuple y$ for some $X \not = \emptyset$. This is in essence the reason why we have to ``enforce'' the empty team property at this stage of the proof.} Thus $\E$ is also downwards closed.

Now since $\D$ is strongly first order $\exists \tuple v \D \tuple v$ is equivalent to some first order sentence $\chi$ over the empty signature, and $\M \models \chi$ iff there is some $R$ such that $(M, R) \in \D$. 
Then consider $\F = \{(M, R): M \models \chi \text{ and } (M, R) \in \E\}$. $\F$ is strongly first order, as any sentence $\phi \in \FO(\F)$ in which $\F$ occurs is equivalent to $\chi \wedge \phi[\E/\F]$, where $\phi[\E/\F]$ is the result of replacing every instance $\F\tuple t$ of $\F$ with $\E \tuple t$; $\F$ is still downwards closed; and $(M, R) \in \F$ iff there is some $S \supseteq R$ such that $(M, S) \in \D$. Therefore $\Dmax = \Fmax$, as required..

Note furthermore that if $\D$ is relativizable then so is $\E$, because $\E^{(P)} \tuple x$ is equivalent to \\$\bot \sqcup (\forall p q \exists \tuple y ( (p \not = q \vee \tuple y = \tuple x) \wedge \D^{(P)}\tuple y))$, and so is $\F$ as well, because any sentence $\phi \in \FO(\F^{(P)})$ in which $\F^{(P)}$ occurs is equivalent to\footnote{Here $\chi^{(P)}$ is the usual relativization of the first order sentence $\chi$ with respect to $P$.} $\chi^{(P)} \wedge \phi[\E^{(P)}/\F^{(P)}]$ and $\E$ is relativizable and strongly first order.
\end{proof}
%\begin{proposition}
%Let $\D$ be a strongly first order dependency. Then the dependency 
%\[
%\E \tuple x := \forall p q \exists \tuple y ( (p \not = q \vee \tuple y = \tuple x) \wedge \D(\tuple y))
%\]
%is strongly first order, is true in $(M, R)$ if and only if there exists some $S \supseteq R$ such that $(M, S) \in \D$, and is such that $\Dmax = \Emax$. 
%\label{propo:emax}
%\end{proposition}
%\begin{proof}
%    Since $\E$ is definable in terms of $\D$ and $\D$ is strongly first order, so is $\E$. The fact that $(M, R) \in \E$ iff there is some $S \supseteq R$ such that $(M, S) \in \D$ can be verified by checking the satisfaction conditions. The fact that $\Emax = \Dmax$ then follows at once. 
%    %But $\Emax = \Dmax$. Indeed, suppose that $(M, R) \in \Dmax$: then $(M, R) \in \D$, and therefore $(M, R) \in \E$. But if $R' \supsetneq R$ then $(M, R') \not \in \E$, since otherwise there would be some $S \supseteq R' \supsetneq R$ such that $(M, S) \in \D$, which is impossible because $R$ is maximal among the relations in $\D$. So $(M, R) \in \Emax$, as required. 
%
%%Conversely, suppose that $(M, R) \in \Emax$. Then $(M, R) \in \E$, and therefore there exists some $S \supseteq R$ such that $(M, R) \in \D$. But then it must be the case that $R=S$, since otherwise $(M, S)$ would be in $\E$ as well and $R$ would not be maximal among the relations in $\E$. So $(M, R) \in \D$. Now if $R' \supsetneq R$, it cannot be that $(M, R') \in D$, since otherwise $(M, R') \in \E$ which again would contradict the fact that $(M, R) \in \Emax$. Therefore $(M, R) \in \Dmax$, as required. 
%\end{proof}
The generalizations of the previous results then follow at once:% from Proposition \ref{propo:emax} and Proposition \ref{propo:dmax} and Theorem \ref{thm:dmaxdef} respectively:
\begin{proposition}
Let $\D$ be a strongly first order dependency. Then $\Dmax$ is also  strongly first order, and whenever $(M, R) \in \D$ there is some $R' \supseteq R$ such that $(M, R') \in \Dmax$.
\end{proposition}
\begin{theorem}
	Let $\D$ be a strongly first order, relativizable dependency. Then there are first order formulas $\theta_1(\tuple x, \tuple z)$ \ldots $\theta_n(\tuple x, \tuple z)$ over the empty signature such that, for all models $\M = (M, R)$, 
	\[
		\M \models \Dmax(R) \Rightarrow \M \models \bigvee_{i=1}^n (\exists \tuple z \forall \tuple x (R \tuple x \leftrightarrow \theta_i(\tuple x, \tuple z)))
	\]
	\label{thm:emaxdef}
\end{theorem}
For our next lemma, we need some model-theoretic machinery: 
\begin{definition}[$\omega$-big models --- \cite{hodges97b}, \S 8.1]
A model $\A$ is $\omega$-big if for all finite tuples $\tuple a$ of elements in $\A$ and for all models $\B$ with the same signature as $\A$, if $\tuple b$ is such that\footnote{Here we write $(\mathfrak A, \tuple a)$ for the model obtained by adding to the signature of $\mathfrak A$ a tuple of new constant symbols $\tuple c = c_1 \ldots c_n$, interpreting them as $\tuple a = a_1 \ldots a_n$. The expression $(\mathfrak B, \tuple b)$ is to be interpreted similarly, for the same choice of new constant symbols. Similarly, an expression of the form $(\mathfrak A, R, \tuple a)$ adds also a new relation symbol to the signature, interpreted as the relation $R$, and so forth.} $(\mathfrak A, \tuple a) \equiv (\mathfrak B, \tuple b)$ and $S$ is a relation over $\B$ then we can find a relation $R$ over $\mathfrak A$ such that $(\mathfrak A, R, \tuple a) \equiv (\mathfrak B, S, \tuple b)$. 
\end{definition}
\begin{definition}[$\omega$-saturated models --- \cite{hodges97b}, \S 8.1]
A model $\M$ is $\omega$-saturated if it realizes all complete $1$-types with respect to $\M$ over any finite parameter set.\footnote{All the details can be found in \cite{hodges97b}, or in other model theory textbooks. Very briefly, a complete $1$-type with respect to $\M$ over some finite parameter set $\{m_1 \ldots m_n\} \subseteq M$ is a set of formulas $\Phi$ of the form $\{\phi(m', m_1 \ldots m_n) : \mathfrak N \models \phi(m', m_1 \ldots m_n)\}$ for some elementary extension $\mathfrak N \succeq \M$ and some element $m'$ of $\mathfrak N$, where $\phi$ ranges over all first order formulas with $n+1$ free variables; and $\M$ realizes such a type if there is some element $m_0 \in M$ such that $\M \models \phi(m_0, m_1 \ldots m_n)$ for all $\phi \in \Phi$. In other words, a complete $1$-type with respect to $\M$ over $m_1 \ldots m_n$ is a description (in terms of first order formulas with parameters in $m_1 \ldots m_n$) of some element that exists in some elementary extension of $\M$; and if $\M$ is $\omega$-saturated, any such description also describes some element that exists in $\M$ itself.}

\end{definition}
The three following results can be found in \cite{hodges97b}:\footnote{We report here only the parts of the results that are relevant for this work.}
\begin{theorem}[\cite{hodges97b}, Theorem 8.1.2]
If a model $\M$ is $\omega$-big then it is $\omega$-saturated. 
\label{thm:ifbigthensat}
\end{theorem}
\begin{theorem}[\cite{hodges97b}, Theorem 8.2.1]
Every model has a $\omega$-big elementary extension. 
\label{thm:bigexists}
\end{theorem}
\begin{theorem}[\cite{hodges97b}, Lemma 8.3.4]
Let $\A$ and $\B$ be $\omega$-saturated structures over a finite signature (containing some relation symbol $R$) such that, for all first order sentences $\psi^+$ in which $R$ occurs only positively, $\A \models \psi^+ \Rightarrow \B \models \psi^+$. Then there are elementary substructures $\mathfrak C$ and $\mathfrak D$ of $\mathfrak A$ and $\mathfrak B$ respectively and a bijective homomorphism $\mathfrak h: \mathfrak C \rightarrow \mathfrak D$ that fixes all symbols except $R$. 
\label{thm:lyndon}
\end{theorem}
The following lemma shows that, in the case of $\omega$-big models, any relation $R$ that satisfies a strongly first order, relativizable, non-jumping dependency $\D$ must satisfy some sentence $\eta(R)$, having a certain specific form, that entails $\D(R)$ over \emph{all} models:
\begin{lemma}
	Let $\D$ be a strongly first order, relativizable, non-jumping dependency, and suppose that $(M, R)$ is a $\omega$-big model such that $(M, R) \in \D$. Then there exist a formula $\psi^+(R, \tuple z)$ over the signature $\{R\}$, positive in $R$, and a formula $\theta(\tuple x, \tuple z)$ over the empty signature such that 
	\begin{enumerate}
		\item $(M, R) \models \exists \tuple z(\psi^+(R, \tuple z) \wedge \forall \tuple x(R \tuple x \rightarrow \theta(\tuple x, \tuple z)))$; 
		\item $\exists \tuple z(\psi^+(R, \tuple z) \wedge \forall \tuple x(R \tuple x \rightarrow \theta(\tuple x, \tuple z))) \models \D(R)$. 
	\end{enumerate}
	\label{lemma:bigpsitheta}
\end{lemma}
\begin{proof}
\begin{figure}
\begin{center}
\includegraphics{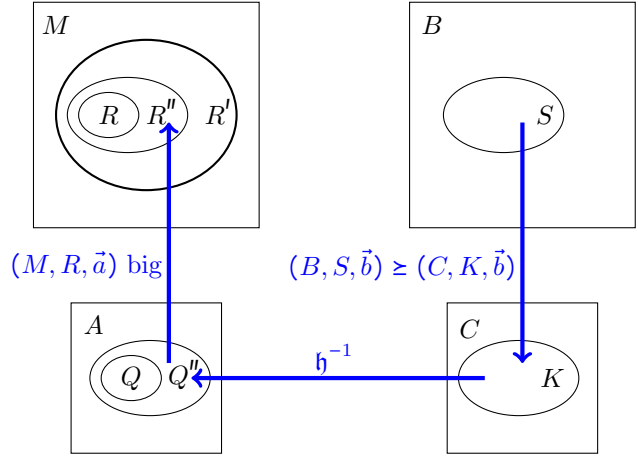}
\end{center}
\caption{The key steps of the proof of Lemma \ref{lemma:bigpsitheta}. Every relation in $M$ between $R$ and $R'$ must satisfy $\D$; but $(B, S) \not \in \D$, so $(C, K) \not \in \D$, so $(A, Q'') \not \in \D$. But $(A, Q, \tuple a) \equiv (M, R, \tuple a)$, $Q \subseteq Q''$, all tuples $\tuple m \in Q''$ satisfy $\theta(\tuple m, \tuple a)$ and $(M, R, \tuple a)$ is $\omega$-big, so there is some $R''$ between $R$ and $R'$ such that $(M, R'') \not \in \D$.}
\label{figure1}
\end{figure}
Since $\D$ is non-jumping, we can find a $R'$ such that $R \subseteq R'$, $(M, R') \in \Dmax$, and $(M, R'') \in \D$ for all $R''$ such that $R \subseteq R'' \subseteq R'$. But then by Theorem \ref{thm:emaxdef} there exist a first order formula $\theta(\tuple x, \tuple z)$ over the empty signature and a tuple of elements $\tuple a$ such that $R' = \{\tuple m \in M^k: (M,R) \models \theta(\tuple m, \tuple a)\}$. 

Now consider $\Psi = \{\eta^+(R, \tuple a): R \text{ occurs only positively in } \eta^+ \text{ and } (M,R) \models \eta^+(R, \tuple a)\}$. 
	
I state that $\Psi \cup \{\forall \tuple x(R \tuple x \rightarrow \theta(\tuple x, \tuple a)), \lnot \D(R)\}$ is unsatisfiable. Indeed, suppose that it is satisfiable, and let $\B = (B, S, \tuple b)$ be a model that satisfies it. By Theorems \ref{thm:ifbigthensat} and \ref{thm:bigexists}, we can assume that $\B$ is $\omega$-saturated. 
 
 Now, since every formula positive in $R$ that is true of $(M, R, \tuple a)$ is also true of $(B, S, \tuple b)$ and both are $\omega$-saturated\footnote{It is trivial to see that if $\mathfrak A$ is $\omega$-big or $\omega$-saturated and $\tuple a$ is a finite tuple of constants then so is $(\mathfrak A, \tuple a)$.}, by Theorem \ref{thm:lyndon} we have that there exist elementary substructures $(A, Q, \tuple a)$ and $(C, K, \tuple b)$ of $(M, R, \tuple a)$ and $(B, S, \tuple b)$ respectively such that $C$ is the image of a bijective homomorphism $\mathfrak h: (A, Q, \tuple a) \rightarrow (C, K, \tuple b)$ that sends $\tuple a$ into $\tuple b$. Now let $Q'' = \mathfrak h^{-1}(K)$ be the inverse image of $K$ under this bijective homomorphism: then $(A, Q'', \tuple a)$ is isomorphic to $(C, K, \tuple b)$, and thus $(A, Q'', \tuple a) \models \forall \tuple x(Q'' \tuple x \rightarrow \theta(\tuple x, \tuple a)) \wedge \lnot \D(Q'')$; and furthermore, since $\mathfrak h$ is a homomorphism, we have at once that $Q \subseteq Q''$. 
 
 Therefore, the model $(A, Q, \tuple a)$ can be expanded to a model $(A, Q, Q'', \tuple a)$ such that $Q \subseteq Q''$, $\forall \tuple x(Q'' \tuple x \rightarrow \theta(\tuple x, \tuple a))$ and $\lnot \D(Q'')$. But $(A, Q, \tuple a)$ is elementarily equivalent to $(M, R, \tuple a)$, which is $\omega$-big. Therefore $(M, R, \tuple a)$ can also be expanded to some $(M, R, R'', \tuple a)$ which is elementarily equivalent to $(A, Q, Q'', \tuple a)$ and in which thus $R''$ likewise contains $R$, contains only tuples $\tuple m $ such that $\theta(\tuple m, \tuple a)$ (and, therefore, is contained in $R'$), and does not satisfy $\D(R'')$. This is however impossible, because we said that no such $R''$ exists; and therefore $\Psi \cup \{\forall \tuple x(R \tuple x \rightarrow \theta(\tuple x, \tuple a)), \lnot \D(R)\}$ is indeed unsatisfiable. Figure \ref{figure1} illustrates the key steps of this argument.

 By compactness, this implies that there exists a finite $\Psi_0 \subseteq \Psi$ such that, for $\psi^+ = \bigwedge \Psi_0$, 
 \begin{enumerate}
     \item $M, R, \tuple a \models \psi^+(R, \tuple a) \wedge \forall \tuple x (R \tuple x \rightarrow \theta(\tuple x, \tuple a))$; 
     \item $\exists \tuple z(\psi^+(R, \tuple z) \wedge \forall \tuple x(R \tuple x \rightarrow \theta(\tuple x, \tuple z))) \models \D(R)$. 
 \end{enumerate}
\end{proof}
Then, exploiting the fact that by Theorem \ref{thm:bigexists} any model is elementarily equivalent to some $\omega$-big model, we can show that $\D$ itself is equivalent (over \emph{all} models) to some first order sentence of a particular form: 
\begin{theorem}
	Let $\D$ be a strongly first order, relativizable, non-jumping dependency. Then there exists a formula $\psi^+(R, \tuple z)$ over the signature $\{R\}$, positive in $R$, and a formula $\theta(\tuple x, \tuple z)$ over the empty signature such that, for all $M$ and $R$, 
	\[
		(M, R) \in \D \Leftrightarrow (M, R) \models \exists \tuple z(\psi^+(R, \tuple z) \wedge \forall \tuple x(R \tuple x \rightarrow \theta(\tuple x, \tuple z)).
	\]
	\label{thm:psitheta}
\end{theorem}
\begin{proof}
	For all countable $M$ and all $R \subseteq M^k$ such that $(M, R) \in \D$, let $(M_1, R_1)$ be a $\omega$-big  elementary extension of it. Then $(M_1, R_1) \in \D$ as well, since $\D$ is first order definable, and therefore by Lemma \ref{lemma:bigpsitheta} there exist some formulas 
	$\psi^+_{M, R}$ and $\theta_{M, R}$ (R positive in $\psi^+$, not appearing in $\theta$) such that
	\begin{enumerate}
	    \item $(M_1, R_1)$ (and therefore $(M, R)$ as well, since it is elementarily equivalent to it) satisfies \\$\exists \tuple z(\psi^+_{M,R}(R, \tuple z) \wedge \forall \tuple x(R \tuple x \rightarrow \theta_{M,R}(\tuple x, \tuple z)))$;
	    \item $\exists \tuple z(\psi^+_{M,R}(R, \tuple z) \wedge \forall \tuple x(R \tuple x \rightarrow \theta_{M,R}(\tuple x, \tuple z))) \models \D(R)$. 
	\end{enumerate}
	
	Now let  
	\[
		T = \{\D(R)\} \cup \{\lnot \exists \tuple z (\psi^+_{M, R}(R, \tuple z) \wedge \forall \tuple x (R \tuple x \rightarrow \theta_{M, R}(\tuple x, \tuple z))): (M, R) \in \D, M \text{ countable}\}.
	\]
	This theory is unsatisfiable: indeed, otherwise by L\"owenheim-Skolem it would have a countable model $(M, R)$, and since $(M, R) \in \D$ we would have that $(M, R) \models \exists \tuple z (\psi^+_{M, R}(R, \tuple z) \wedge \forall \tuple x (R \tuple x \rightarrow \theta_{M, R}(\tuple x, \tuple z))$. Thus $(M, R)$ would not be a model of $T$, contradicting our hypothesis. 

	So by compactness there exist formulas $\psi_i^+(R, \tuple z_i), \theta_i(\tuple x, \tuple z_i)$, $i=1 \ldots n$  such that
	\[
		\D(R) \equiv \bigvee_{i=1}^n \exists \tuple z_i(\psi_i^+(R, \tuple z_i) \wedge \forall \tuple x(R \tuple x \rightarrow \theta_i(\tuple x, \tuple z_i))). 
	\]
	
	But then $\D(R)$ is also equivalent to 
	\[
		\exists \tuple q_1 \ldots \tuple q_n \tuple p \tuple z_1 \ldots \tuple z_n \left( \psi^+_0(R, \tuple q, \tuple p,\tuple z_1 \ldots \tuple z_n)
		 \wedge 
		\forall \tuple x \left((R \tuple x \rightarrow \theta_0(\tuple x, \tuple q, \tuple p, \tuple z_1 \ldots \tuple z_n)\right)
		\right)
	\]
	for $\psi^+_0 = \left(
		\bigwedge_{i\not =j} \tuple q_i \not = \tuple q_j \wedge 
		\bigvee_i \tuple p = \tuple q_i \wedge 
		\bigwedge_i (\tuple p = \tuple q_i \rightarrow \psi_i^+(R, \tuple z_i))\right)$ and $\theta_0 = \bigwedge_i (\tuple p = \tuple q_i \rightarrow \theta_i(\tuple x, \tuple z_i))$, where $\tuple p$ and all $\tuple q_i$ are tuples of distinct, new variables of length $\lceil \log_2(n)\rceil$. 
\end{proof}
We are now almost done. All that's left to do is to show that any first order dependence that is equivalent to some sentence of the above form is definable in terms of global disjunctions, constancy atoms, and first order upwards closed dependencies:
\begin{corollary}
	Every strongly first order, relativizable, non-jumping dependency is definable in terms of first order upwards closed dependencies, constancy atoms and global disjunctions. 
\end{corollary}
\begin{proof}
	Let $\D$ be such a dependency of arity $k$. By the previous theorem, $\D(Q)$ is equivalent to some expression of the form $\exists \tuple z (\psi^+(Q, \tuple z) \wedge \forall \tuple x (Q \tuple x \rightarrow \theta(\tuple x, \tuple z)))$, where the $k$-ary variable $Q$ occurs only positively in $\psi^+$ and not at all in $\theta$. Now, if $l$ is the length of $\tuple z$, consider the $(k+l)$-ary first order upwards closed dependency $\E$ 
%such that 
%		$\M \models_X \E \tuple x \tuple z \Leftrightarrow \exists \tuple a \in X(\tuple z) \text{ s.t. } \psi^+(X(\tuple x), \tuple a)$, 
	such that, for any model $\M$ and any team $X$, $\M \models_X \E \tuple x \tuple z$ if and only if $X \not = \emptyset$ and there is some tuple $\tuple a \in X(\tuple z)$ such that $\psi^+(X(\tuple x), \tuple a)$. This $\E$ is upwards closed (since the relation symbol appears only positively in $\psi^+$), and it is first order, because 	
	\[
		(M,R) \in \E \Leftrightarrow (M, R) \models \exists \tuple u \tuple v(R \tuple u \tuple v \wedge \psi^+[\exists \tuple w R \tuple t \tuple w/ Q \tuple t](\tuple v))
	\]
	where $|\tuple u| = k$, $|\tuple v| = |\tuple w| = l$ and  $\psi^+[\exists \tuple w R \tuple t \tuple w/Q \tuple t](\tuple v)$ is obtained from $\psi^+(Q, \tuple v)$ by replacing every occurrence $Q \tuple t$ of $Q$ in it (for every $\tuple t$) with $\exists \tuple w R \tuple t \tuple w$. Therefore $\E$ is strongly first order. 

Moreover, let $\F$ be another dependency, of the same arity of $\D$, defined as 
\[
	(M, R) \in \F \Leftrightarrow M \models \exists \tuple v \psi^+(\emptyset, \tuple v).
\]
Then $\F$ is trivially upwards closed (in fact, the relation $R$ does not affect membership or non-membership of $(M, R)$ in $F$) and first order, and hence it is also strongly first order.

Now, $\D\tuple x$ is definable in $\FO(\E, \F, =\!\!(\cdot), \sqcup)$ as 
\begin{equation}
	(\F \tuple x \wedge \bot) \sqcup (\exists \tuple z (=\!\!(\tuple z) \wedge \E \tuple x \tuple z \wedge \theta(\tuple x, \tuple z))).
	\label{eq:defD}
\end{equation}

Indeed, suppose that some team $X$ satisfies  (\ref{eq:defD}) in some model $\M$. Then either $\M \models_X \F \tuple x \wedge \bot$ or $\M \models_X \exists \tuple z (=\!\!(\tuple z) \wedge \E \tuple x \tuple z \wedge \theta(\tuple x, \tuple z))$. In the first case, $X = \emptyset$, because $\M \models_X \bot$, and $\M \models \psi^+(\emptyset, \tuple a)$ for some $\tuple a$, because $(M, \emptyset) \in \F$. Furthermore, for $Q = X(\tuple x) = \emptyset$ we have trivially that $\M \models \forall \tuple x(Q  \tuple x \rightarrow \theta(\tuple x, \tuple a))$, and hence $(M, X(\tuple x)) = (M, \emptyset) \in \D$ as required. 

	In the second case, instead, there exists some tuple $\tuple a$ such that $\M \models_{X[\tuple a/\tuple z]} \E \tuple x \tuple z \wedge \theta(\tuple x, \tuple z)$. Since $\M \models_{X[\tuple a/\tuple z]} \E \tuple x \tuple z$, $X \not = \emptyset$ and there is some possible value $\tuple a' \in X[\tuple a/\tuple z](\tuple z)$ such that, for $Q = X(\tuple x)$, $(M, Q) \models \psi^+(Q, \tuple a')$. But $\tuple z$ takes only the value $\tuple a$ in $X[\tuple a/\tuple z]$, and hence $\tuple a' = \tuple a$. Furthermore, since $\M \models_{X[\tuple a/\tuple z]} \theta(\tuple x, \tuple z)$, by Proposition \ref{propo:FOL} we have that for every $s \in X[\tuple a/\tuple z]$ $\M \models_s \theta(\tuple x, \tuple z)$. But for every $\tuple b \in Q$ there is some $s \in X[\tuple a/\tuple z]$ with $s(\tuple x) = \tuple b$ and $s(\tuple z) = \tuple a$, and so $(M, Q) \models \forall \tuple x (Q \tuple x \rightarrow \theta(\tuple x, \tuple a))$; and therefore, $(M, Q) \in \D$ and $\M \models_X \D \tuple x$ as required.

Conversely, suppose that $(M, Q) \in \D$ for $Q = X(\tuple x)$. This implies that there exists some $\tuple a$ such that $\psi^+(Q, \tuple a) \wedge \forall \tuple x (Q \tuple x \rightarrow \theta(\tuple x, \tuple a))$. If $Q = \emptyset$ then $X = \emptyset$ as well and $(M, \emptyset) \in \F$, which implies that $\M \models_\emptyset \F \tuple x \wedge \bot$ and that (\ref{eq:defD}) holds. If instead $X$ and $Q$ are nonempty, let $Y  = X[\tuple a / \tuple z]$. Then $\M \models_Y =\!\!(\tuple z)$, because $\tuple z$ is constant in $Y$; $\M \models_Y \E \tuple x\tuple z$, because for $Q = X(\tuple x) = Y(\tuple x)$ we have that $\psi^+(Q, \tuple a)$ and $Y(\tuple z) = \{\tuple a\}$; and, for all $\tuple b \in Q$, we have that $\theta(\tuple b, \tuple a)$ and hence $\M \models_Y \theta(\tuple x, \tuple z)$. Hence again (\ref{eq:defD}) holds, and this concludes the proof.
\end{proof}

The above result provides a full characterization of strongly first order dependencies that are relativizable and non-jumping. I suspect that this result may be further generalized to jumping dependencies as a consequence of the following\\

\begin{conjecture}
Every strongly first order dependency $\D(R)$ can be expressed as a disjunction $\bigvee_i \D_i(R)$ of dependencies $\D_i$ that are strongly first order \emph{and} non-jumping.
\label{conj:3}
\end{conjecture}

\noindent Furthermore, if Conjecture \ref{conj:2} holds, the requirement of relativizability may be also disposed of. %This however, is left for future work. 
\section{Conclusions and Further Work}
In this work I provided a full characterization of strongly first order dependencies in Team Semantics under the two (commonly true) additional assumptions that these dependencies are relativizable and non-jumping; and, in doing so, I disproved an earlier conjecture regarding a general characterization of strongly first order dependencies and highlighted the importance of the global disjunction connective in the study of Team Semantics. The obvious next step consists in trying to find ways to remove or weaken these assumptions, for instance by proving Conjectures \ref{conj:2} and/or \ref{conj:3}. Another research direction worth investigating at this point is to generalize this approach to the study of \emph{operators} (rather than just \emph{dependencies}) in Team Semantics, building on the work on generalized quantifiers in Team Semantics of \cite{barbero19,engstrom12,kuusisto2015}. 

It would also be interesting to study axiomatizations for First Order Logic plus strongly first order dependencies. The procedure to translate from $\FO(\mathcal D^\uparrow, =\!\!(\cdot))$ to $\FO$ described in \cite{galliani2015upwards} is deterministic and ends in finitely many steps for all formulas, and therefore if Conjecture \ref{conj:1bis} holds then it should be possible to extend the proof system for logics based on Team Semantics of \cite{luck2018axioms} or of \cite{engstrom2017dependence} to deal with all such dependencies. 

Finally, the question of whether and to which degree these result generalize to probabilistic variants of Team Semantics is almost entirely open. Such variants allow an even larger variety of possible choices of connectives and atoms than non-probabilistic Team Semantics does; and, therefore, the classification of logics based on Probabilistic Team Semantics promises to be an intriguing and highly nontrivial area of research.
\bibliographystyle{eptcs}
\bibliography{biblio}
\end{document}